\UseRawInputEncoding
\documentclass[12pt]{amsart}
\usepackage{}

\usepackage{cancel}
\usepackage{amsmath}
\usepackage{amsfonts}
\usepackage{amssymb}
\usepackage[all]{xy}           

\usepackage{bbding}
\usepackage{txfonts}
\usepackage{amscd}
\usepackage{mathrsfs}

\usepackage[shortlabels]{enumitem}
\usepackage{ifpdf}
\ifpdf
  \usepackage[colorlinks,final,backref=page,hyperindex]{hyperref}
\else
  \usepackage[colorlinks,final,backref=page,hyperindex,hypertex]{hyperref}
\fi
\usepackage{tikz}
\usepackage[active]{srcltx}

\makeatletter

\newtheorem{thm}{Theorem}[section]

\newtheorem{cor}[thm]{Corollary}
\newtheorem{pro}[thm]{Proposition}
\newtheorem{ex}[thm]{Example}
\newtheorem{rmk}[thm]{Remark}
\newtheorem{defi}[thm]{Definition}

\newcommand {\emptycomment}[1]{}
\newcommand {\yh}[1]{{\marginpar{*}\scriptsize\textcolor{purple}{yh: #1}}}

\newcommand{\lon }{\,\rightarrow\,}
\newcommand{\be }{\begin{equation}}
\newcommand{\ee }{\end{equation}}

\newcommand{\g}{\mathfrak g}
\newcommand{\h}{\mathfrak h}

\newcommand{\huaG}{\mathcal{G}}

\newcommand{\Ri}{\mathsf{R}}

\newcommand{\huaU}{\mathcal{U}}

\newcommand{\huaX}{\mathcal{X}}

\newcommand{\huaC}{{\mathcal{C}}}

\newcommand{\InnDer}{\mathrm{InnDer}}
\newcommand{\Orb}{\mathrm{Orb}}
\newcommand{\InnAut}{\mathrm{InnAut}}
\newcommand{\GL}{\mathrm{GL}}

\newcommand{\Id}{{\rm{Id}}}

\newcommand{\br}[1]{   [ \cdot,    \cdot  ]   }

\newcommand{\Hom}{\mathrm{Hom}}

\newcommand{\Der}{\mathrm{Der}}

\newcommand{\Aut}{\mathrm{Aut}}

\newcommand{\ad}{\mathrm{ad}}

\begin{document}

\title{Stability and rigidity of $3$-Lie algebra morphisms}

\author{Jun Jiang}
\address{Department of Mathematics, Jilin University, Changchun 130012, Jilin, China}
\email{junjiang@jlu.edu.cn}

\author{Yunhe Sheng}
\address{Department of Mathematics, Jilin University, Changchun 130012, Jilin, China}
\email{shengyh@jlu.edu.cn}

\author{Geyi Sun}
\address{Department of Mathematics, Jilin University, Changchun 130012, Jilin, China}
\email{sungy21@mails.jlu.edu.cn}


\begin{abstract}
In this paper, first we use the higher derived brackets to construct an $L_\infty$-algebra, whose Maurer-Cartan elements are $3$-Lie algebra morphisms. Using the differential in the $L_\infty$-algebra that govern deformations of the morphism, we give the cohomology of a $3$-Lie algebra morphism. Then we  study the rigidity and stability of $3$-Lie algebra morphisms using the established cohomology theory. In particular, we show that if the first cohomology group is trivial, then the morphism is rigid; if the second cohomology group is trivial, then the morphism is stable. Finally, we study the stability of $3$-Lie subalgebras similarly.
\end{abstract}


\keywords{}

\keywords{3-Lie algebra, morphism, cohomology, deformation}

\renewcommand{\thefootnote}{}
\footnotetext{2020 Mathematics Subject Classification. 17A42, 17B40, 17B56
}
\maketitle
\tableofcontents

\allowdisplaybreaks


\section{Introduction}

For an algebraic, or a geometric structure, a meaningful problem is that describe a neighborhood of such structure in its moduli space. Deforming the structure is a useful way to solve this problem. Roughly speaking, a deformation of a structure is a small curve through the original structure in the moduli space. The concept of a formal deformation of an algebraic structure began with the seminal work of Gerstenhaber~\cite{Ge0,Ge} for associative algebras. Nijenhuis and Richardson   extended this study to Lie algebras~\cite{NR,NR2}.  There is a well known slogan, often attributed to Deligne, Drinfeld and Kontsevich: every reasonable deformation theory is controlled by a differential graded Lie algebra, determined up to quasi-isomorphism. This slogan has been made into a rigorous theorem by Lurie and Pridham \cite{Lu,Pr}. Rigidity and stability are natural questions need to be considered in studying deformations. As important invariants, the cohomology is a useful tool to characterize deformations. In general, taking the differential of a deformation, we obtain a cocycle. Moreover,  rigidity and stability can be characterized in terms of the cohomology groups.

 It is also meaningful to study deformations of morphisms of given algebraic structures. Deformations of morphisms of associative algebras were studied  in \cite{Borisov, Bor1}, and deformations of morphisms of Lie algebras were studied  in \cite{Das,Fre,Fregier-Zambon-2,NR0}. On the other hand, morphisms between two algebraic structures naturally give rise to subalgebras of the direct sum algebras. Deformations of subalgebras of Lie algebras were study in \cite{CSS,Ri}.

The purpose of this paper is to study deformations of 3-Lie algebra morphisms, with special attention to the rigidity and stability problems.
The notion of  $3$-Lie algebras, or more generally,   $n$-Lie algebras (also called   Filippov algebras)~\cite{FI} can be seen as a generalization of Lie algebras to higher arities.  See the review article~\cite{deI} for their applications in mathematical physics. In particular, $n$-Lie algebras are the algebraic structures corresponding to Nambu mechanics and related to Nambu-Poisson structures  \cite{Nam,SJ,Ta}. Moreover, $3$-Lie algebras also appear in the study of Bagger-Lambert-Gustavsson theory of multiple M2-branes and string theory~\cite{BL,CS,deM}.  Deformations of $3$-Lie algebras were studied in~\cite{Fa, Mak, Ta1}.
In~\cite{ABM}, Arfa, Ben Fraj and Makhlouf introduced the cohomology complex of $n$-Lie algebra morphisms and studied one-parameter formal deformations of $n$-Lie algebras morphisms.

In this paper,  first  we use the derived bracket approach to give the controlling algebra of $3$-Lie algebra morphisms. Consequently, the cohomology of a $3$-Lie algebra morphism is introduced using the differential in the $L_\infty$-algebra that govern deformations of the morphism. Then we study (geometric) deformations of $3$-Lie algebra morphisms, which is different from formal deformations considered in \cite{ABM}. In particular, rigidity and stability of deformations of $3$-Lie algebra morphisms are explored. We show that for a 3-Lie algebra morphism $f:\g\to \h$, if $H^1(f)=0,$ then it is rigid; if  $H^2(f)=0,$ then it is stable. Due to the close relation between morphisms and subalgebras, we also study deformations of $3$-Lie subalgebras similarly.

The paper is organized as follows. In Section \ref{sec:con}, we construct an $L_\infty$-algebra via the higher derived brackets, whose Maurer-Cartan elements are $3$-Lie algebra morphisms (Theorem \ref{Maurer-Cartanm}). In particular, we give the $L_\infty$-algebra that governs deformations of $3$-Lie algebra morphisms using Getzler's twisted brackets. In Section \ref{sec:coh}, we introduce the cohomology of $3$-Lie algebra morphisms using the differential in the $L_\infty$-algebra that governs deformations of $3$-Lie algebra morphisms. In Section \ref{sec:sec:rig-hom}, (geometric) deformations of $3$-Lie algebra morphisms are studied. We give the conditions when $3$-Lie algebra morphisms are rigid and stable (Theorem \ref{thmrigidity}, Theorem \ref{thmrigidity2}). We also give the sufficient condition on a cocycle giving a deformation of a $3$-Lie algebra morphism (Theorem \ref{thmnedef}). In Section \ref{sec:rig-sub}, we study deformations of $3$-Lie subalgebras with special attention to the stability (Theorem \ref{staoflies}). We show that deformations of $3$-Lie algebra morphisms give rise to deformations of $3$-Lie subalgebras (Proposition \ref{deformal}), and establish the relation between the corresponding cohomology groups.

\section{Derived brackets and Maurer-Cartan characterization of $3$-Lie algebra morphisms}\label{sec:con}

In this section, first we recall derived brackets which is a very useful way to construct $L_\infty$-algebras. Then we construct the controlling algebra of $3$-Lie algebra morphisms using derived brackets.
\subsection{$L_\infty$-algebras and derived brackets}

Let $\g=\oplus_{k\in\mathbb Z}\g^k$ be a $\mathbb Z$-graded vector space.
 The   desuspension operator  $s^{-1}$ changes the grading of $\g$ according to the rule $(s^{-1}\g)^i:=\g^{i+1}$. The  degree $-1$ map $s^{-1}:\g\lon s^{-1}\g$ is defined by sending $v\in \g$ to its   copy $s^{-1}v\in s^{-1}\g$.

\begin{defi}\cite{KS}
An $L_\infty$-algebra is a $\mathbb Z$-graded vector space $\g=\oplus_{k\in\mathbb Z}\g^k$ equipped with a collection $(k\ge 1)$ of linear maps $l_k:\otimes^k\g\lon\g$ of degree $1$ with the property that, for any homogeneous elements $x_1,\cdots,x_n\in \g$, we have
\begin{itemize}\item[\rm(i)]
{\em (graded symmetry)} for every $\sigma\in S_{n}$,
\begin{eqnarray*}
l_n(x_{\sigma(1)},\cdots,x_{\sigma(n)})=\varepsilon(\sigma)l_n(x_1,\cdots,x_n),
\end{eqnarray*}
\item[\rm(ii)] {\em (generalized Jacobi identity)} for all $n\ge 1$,
\begin{eqnarray*}\label{sh-Lie}
\sum_{i=1}^{n}\sum_{\sigma\in  S(i,n-i) }\varepsilon(\sigma)l_{n-i+1}(l_i(x_{\sigma(1)},\cdots,x_{\sigma(i)}),x_{\sigma(i+1)},\cdots,x_{\sigma(n)})=0,
\end{eqnarray*}

\end{itemize}where $\varepsilon(\sigma)=\varepsilon(\sigma;v_1,\cdots,v_n)$ is the   Koszul sign for a permutation $\sigma\in S_n$ and $v_1,\cdots, v_n\in V$.
\end{defi}

\begin{defi}\cite{Fregier-Zambon-1}
 A Maurer-Cartan element of an $L_\infty$-algebra $(\g,\{l_k\}_{k=1}^{+\infty})$ is an element $\alpha\in \g^0$ satisfying the Maurer-Cartan equation
\begin{eqnarray}\label{MC-equation}
\sum_{k=1}^{+\infty}\frac{1}{k!}l_k(\alpha,\cdots,\alpha)=0.
\end{eqnarray}
\end{defi}

\begin{rmk}
In general, the Maurer-Cartan equation \eqref{MC-equation} makes sense when the $L_\infty$-algebra is a filtered $L_\infty$-algebra \cite{DR}.  In the following,  the $L_\infty$-algebra under consideration satisfies  $l_k=0$ for $k$ sufficiently big, so  the Maurer-Cartan equation makes sense.
\end{rmk}

Let $\alpha$ be a Maurer-Cartan element. Define $l_{k}^{\alpha}:\otimes^{k}\g\lon\g ~~(k\geq1)$ by
\begin{equation}\label{eqdeftw}
l_{k}^{\alpha}(x_1,\cdots,x_k)=\sum_{n=0}^{+\infty}\frac{1}{n!}l_{k+n}(\underbrace{\alpha,\cdots,\alpha}_{n},x_1,\cdots,x_k).
\end{equation}

\begin{thm}\label{twistLin}{\rm(\cite{Get})}
  $(\g,\{l_k^{\alpha}\}_{k=1}^{+\infty})$ is an $L_\infty$-algebra, called the twisted $L_\infty$-algebra.
\end{thm}

Now we recall the V-data and derived brackets, which are very powerful tool to construct explicit $L_\infty$-algebras.

\begin{defi}
A $V$-data consists of a quadruple $(L,F,P,\Delta)$, where
\begin{itemize}
\item[$\bullet$] $(L,[\cdot,\cdot])$ is a graded Lie algebra,
\item[$\bullet$] $F$ is an abelian graded Lie subalgebra of $(L,[\cdot,\cdot])$,
\item[$\bullet$] $P:L\lon L$ is a projection, that is $P\circ P=P$, whose image is $F$ and its kernel is a  graded Lie subalgebra of $(L,[\cdot,\cdot])$,
\item[$\bullet$] $\Delta$ is an element in $  \ker(P)^1$ such that $[\Delta,\Delta]=0$.
\end{itemize}
\end{defi}

\begin{thm}\label{L11}{\rm (\cite{Vo})}
Let $(L,F,P,\Delta)$ be a $V$-data. Then $(F,\{{l_k}\}_{k=1}^{+\infty})$ is an $L_\infty$-algebra where
\begin{eqnarray*}
l_k(a_1,\cdots,a_k)=P\underbrace{[\cdots[[}_k\Delta,a_1],a_2],\cdots,a_k],\quad\mbox{for homogeneous}~   a_1,\cdots,a_k\in F.
\end{eqnarray*}
\end{thm}

Moreover, there is an $L_\infty$-algebra structure on $s^{-1}L\oplus F$ constructed as following.

\begin{thm}\label{defili}{\rm (\cite{Vo,Fregier-Zambon-1})}
Let $(L,F,P,\Delta)$ be a $V$-data. Then the graded vector space $s^{-1}L\oplus F$  is an $L_\infty$-algebra where the nontrivial products, called the higher derived brackets, are given by
\begin{eqnarray*}\label{V-shla-big-algebra}
l_1(s^{-1}f,\theta)&=&(-s^{-1}[\Delta,f],P(f+[\Delta,\theta])),\\
l_2(s^{-1}f,s^{-1}g)&=&(-1)^fs^{-1}[f,g],\\
l_k(s^{-1}f,\theta_1,\cdots,\theta_{k-1})&=&P[\cdots[[f,\theta_1],\theta_2]\cdots,\theta_{k-1}],\quad k\geq 2,\\
l_k(\theta_1,\cdots,\theta_{k-1},\theta_k)&=&P[\cdots[[\Delta,\theta_1],\theta_2]\cdots,\theta_{k}],\quad k\geq 2.
\end{eqnarray*}
Here $\theta,\theta_1,\cdots,\theta_k$ are homogeneous elements of $F$ and $f,g$ are homogeneous elements of $L$. 

Moreover, if $L'$ is a graded Lie subalgebra of $L$ that satisfies $[\Delta,L']\subset L'$, then $s^{-1}L'\oplus F$ is an $L_\infty$-subalgebra of the above $L_\infty$-algebra $(s^{-1}L\oplus F,\{l_k\}_{k=1}^{+\infty})$.
\end{thm}

\subsection{Maurer-Cartan characterization of $3$-Lie algebra morphisms}
\begin{defi}\cite{FI}
A $3$-Lie algebra is a vector space $\g$ with a linear map $\pi:\wedge^{3}\g\lon\g$ such that
$$
\pi(x_1, x_2, \pi(x_3, x_4, x_5))=\pi(\pi(x_1, x_2, x_3), x_4, x_5)+\pi(x_3, \pi(x_1, x_2, x_4), x_5)+\pi(x_3, x_4, \pi(x_1, x_2, x_5)),
$$
where $x_i\in\g, 1\leq i\leq 5.$
\end{defi}

\begin{defi}
Let $(\g, \pi)$ and $(\h, \mu)$ be $3$-Lie algebras. A $3$-Lie algebra morphism is a linear map $f:\g\lon\h$ such that
$$
f(\pi(x, y, z))=\mu(f(x), f(y), f(z)), \quad \forall x, y, z\in\g.
$$
\end{defi}

Let $\g$ be a vector space. Consider the following graded vector space
$$
C^{*}_{\text{3-Lie}}(\g, \g)=\oplus_{n\geq 0}C^{n}(\g, \g)=\oplus_{n\geq 0}\Hom(\underbrace{\wedge^{2}\g\otimes\cdots\otimes\wedge^{2}\g}_n\wedge\g, \g), \quad (n\geq 0).
$$

\begin{thm}\label{MC3}{\rm (\cite{NR bracket of n-Lie})}
With the above notations, $(C^{*}_{\text{\rm{3-Lie}}}(\g, \g), [\cdot, \cdot]_{\Ri})$ is a graded Lie algebra, where
\begin{eqnarray}\label{3-Lie-bracket}
[P,Q]_{\Ri}=P{\circ}Q-(-1)^{pq}Q{\circ}P,\quad \forall~ P\in C^{p}(\g,\g),Q\in C^{q}(\g,\g),
\end{eqnarray}
and $P{\circ}Q\in C^{p+q}(\g,\g)$ is defined by

\begin{small}
\begin{equation*}
\begin{aligned}
&(P{\circ}Q)(\mathfrak{X}_1,\cdots,\mathfrak{X}_{p+q},x)\\
=&\sum_{k=1}^{p}(-1)^{(k-1)q}\sum_{\sigma\in \mathbb S(k-1,q)}(-1)^\sigma P\Big(\mathfrak{X}_{\sigma(1)},\cdots,\mathfrak{X}_{\sigma(k-1)},
Q\big(\mathfrak{X}_{\sigma(k)},\cdots,\mathfrak{X}_{\sigma(k+q-1)},x_{k+q}\big)\wedge y_{k+q},\mathfrak{X}_{k+q+1},\cdots,\mathfrak{X}_{p+q},x\Big)\\
&+\sum_{k=1}^{p}(-1)^{(k-1)q}\sum_{\sigma\in \mathbb S(k-1,q)}(-1)^\sigma P\Big(\mathfrak{X}_{\sigma(1)},\cdots,\mathfrak{X}_{\sigma(k-1)},x_{k+q}\wedge
Q\big(\mathfrak{X}_{\sigma(k)},\cdots,\mathfrak{X}_{\sigma(k+q-1)},y_{k+q}\big),\mathfrak{X}_{k+q+1},\cdots,\mathfrak{X}_{p+q},x\Big)\\
&+\sum_{\sigma\in \mathbb S(p,q)}(-1)^{pq}(-1)^\sigma P\Big(\mathfrak{X}_{\sigma(1)},\cdots,\mathfrak{X}_{\sigma(p)},
Q\big(\mathfrak{X}_{\sigma(p+1)},\cdots,\mathfrak{X}_{\sigma(p+q-1)},\mathfrak{X}_{\sigma(p+q)},x\big)\Big),\\
\end{aligned}
\end{equation*}
\end{small}
 for all $\mathfrak{X}_{i}=x_i\wedge y_i\in \wedge^2 \g$, $~i=1,2,\cdots,p+q$ and $x\in\g.$

  Moreover,  $\pi:\wedge^3\g\longrightarrow\g$ is a $3$-Lie algebra on $\g$ if and only if $[\pi,\pi]_{\Ri}=0$, i.e. $\pi$ is a Maurer-Cartan element of the graded Lie algebra $(C^{*}_{\text{\rm{3-Lie}}}(\g, \g),[\cdot,\cdot]_{\Ri})$.
  \end{thm}

In the following, we use derived brackets to construct an $L_\infty$-algebra whose Maurer-Cartan elements are $3$-Lie algebra morphisms.

Let $\g$ and $\h$ be vector spaces and $\huaG=\g\oplus\h$. Denote by $L$ the graded vector space
$$
L=\mathop{\oplus}\limits_{n\geq0}\Hom(\underbrace{\wedge^{2}\huaG\otimes\cdots\otimes\wedge^{2}\huaG}_n\wedge\huaG, \huaG).
$$
Denote by $\huaG^{l,k}$ the subspace of
$\otimes ^n(\wedge^2(\g\oplus \h))\wedge(\g\oplus \h)$
where the numbers of $\g$ and $\h$ are $l$ and $k$ respectively. Then the vector space $\otimes ^n(\wedge^2(\g\oplus \h))\wedge(\g\oplus \h)$ is
expended into $\mathop{\oplus}\limits_{l + k = 2n + 1}\huaG ^{l,k}$. Thus we have that
$$
L\cong\mathop{\oplus}\limits_{n\geq0}\Big(\mathop{\sum}\limits_{l+k=2n+1}\Hom(\huaG^{l,k},\g)\oplus \Hom(\huaG^{l,k},\h)\Big)
$$
and denote this isomorphism by $\mathrm{H}: \mathop{\oplus}\limits_{n\geq0}\Big(\mathop{\sum}\limits_{l+k=2n+1}\Hom(\huaG^{l,k},\g)\oplus \Hom(\huaG^{l,k},\h)\Big)\lon L$.

For $P\in\Hom(\underbrace{\wedge^{2}\g\otimes\cdots\otimes\wedge^{2}\g}_n\wedge\g,\h)$, we have
\begin{equation}\label{eqdef01}
\mathrm{H}(P)(\mathfrak{X}_{1}, \cdots, \mathfrak{X}_{n}, (x, a))=(0, P(x_1\wedge y_1, \cdots, x_n\wedge y_n, x)),
\end{equation}
where $\mathfrak{X}_{i}=(x_i, a_i)\wedge (y_i, b_i)\in\wedge^{2}(\g\oplus\h)$.

Moreover, using this isomorphism, we can transfer the graded Lie bracket $[\cdot, \cdot]_{\Ri}$ on $L$ to
$$
\mathop{\oplus}\limits_{n\geq0}\Big(\mathop{\sum}\limits_{l+k=2n+1}\Hom(\huaG^{l,k},\g)\oplus \Hom(\huaG^{l,k},\h)\Big),
$$
which is also denoted by $[\cdot, \cdot]_{\Ri}$, i.e.
$$
[A, B]_{\Ri}=\mathrm{H}^{-1}([\mathrm{H}(A), \mathrm{H}(B)]_{\Ri}), \quad \forall A, B\in\mathop{\oplus}\limits_{n\geq0}\Big(\mathop{\sum}\limits_{l+k=2n+1}\Hom(\huaG^{l,k},\g)\oplus \Hom(\huaG^{l,k},\h)\Big).
$$
\begin{pro}\label{control}
Let $(\g, \pi)$ and $(\h, \mu)$ be $3$-Lie algebras. We have a $V$-data $(L, F, P, \Delta)$ as follows:
\begin{itemize}
  \item[$\bullet$] the graded Lie algebra $(L, [\cdot,\cdot]_{\Ri})$ is given by $$(\oplus_{n=0}^{+\infty}\Hom(\underbrace{\wedge^{2}(\g\oplus\h)\otimes\cdots\otimes\wedge^{2}(\g\oplus\h)}_n\wedge(\g\oplus\h),\g\oplus\h), [\cdot,\cdot]_{\Ri});$$
  \item[$\bullet$] the abelian graded Lie subalgebra $F$ is given by $\oplus_{n=0}^{+\infty}\Hom(\underbrace{\wedge^{2}\g\otimes\cdots\otimes\wedge^{2}\g}_n\wedge\g,\h)$;
  \item[$\bullet$] $P:L\lon L$ is the projection onto the subspace $F$;
  \item[$\bullet$] $\Delta=\pi+\mu$.
\end{itemize}
Consequently, we obtain an $L_{\infty}$-algebra $(F, \{l_k\}^{+\infty}_{k=1})$, where $l_{i}$ are given by
\begin{equation}\label{eqabli}
\begin{cases}
\begin{aligned}
l_1(a_1)&=[\pi, a_1]_{\Ri}, \\
l_2(a_1, a_2)&=0, \\
l_3(a_1, a_2, a_3)&=[[[\mu, a_1]_{\Ri}, a_2]_{\Ri}, a_3]_{\Ri}, \\
l_k&=0, \quad k\geq 4.
\end{aligned}
\end{cases}
\end{equation}
\end{pro}
\begin{proof}
For $Q_1, Q_2\in F$, since
$$
\mathrm{H}([Q_1, Q_2]_{\Ri})=([\mathrm{H}(Q_1), \mathrm{H}(Q_2)]_{\Ri})=\mathrm{H}(Q_1){\circ}\mathrm{H}(Q_2)-(-1)^{pq}\mathrm{H}(Q_2){\circ}\mathrm{H}(Q_1)$$
and
$$
\mathrm{Im}(\mathrm{H}(Q_1))\in(0, \h), \quad \mathrm{Im}(\mathrm{H}(Q_2))\in(0, \h),
$$
by \eqref{eqdef01}, we have that $\mathrm{H}([Q_1, Q_2]_{\Ri})=0$, which implies that $F$ is an abelian graded Lie subalgebra.

It is obvious that $\Delta\in\ker(P)^1$. Since $\pi$ and $\mu$ are $3$-Lie algebra structures on $\g$ and $\h$, by Theorem \ref{MC3}, we have
$$[\pi, \pi]_{\Ri}=[\mu, \mu]_{\Ri}=0.$$
Moreover, since
$$
\mathrm{H}([\pi, \mu]_{\Ri})=[\mathrm{H}(\pi), \mathrm{H}(\mu)]_{\Ri}=\mathrm{H}(\pi)\circ \mathrm{H}(\mu)+\mathrm{H}(\mu)\circ \mathrm{H}(\pi),
$$
and
$$
\mathrm{Im}(\mathrm{H}(\pi))\in(\g, 0), \quad \mathrm{Im}(\mathrm{H}(\mu))\in(0, \h),
$$
we have that $[\pi, \mu]_{\Ri}=0$, which implies that $[\Delta, \Delta]_{\Ri}=0$.
Therefore $(L, F, P, \Delta)$ is a $V$-data. By Theorem \ref{L11}, it follows that $(F, \{l_k\}^{+\infty}_{k=1})$ is an $L_\infty$-algebra, where
$$
l_k(a_1,\cdots,a_k)=P\underbrace{[\cdots[[}_k\pi+\mu,a_1],a_2],\cdots,a_k],\quad\mbox{for homogeneous}~   a_1,\cdots,a_k\in F.
$$
On the other hand, it is straightforward to deduce that
\begin{eqnarray*}
[\pi, a_1]_{\Ri}\in F, \quad P([\mu, a_1]_{\Ri})=0,  \quad [[\pi, a_1]_{\Ri}, a_2]_{\Ri}=0,  \quad P([[\mu, a_1]_{\Ri}, a_2]_{\Ri})=0,
\end{eqnarray*}
and $[[[\mu, a_1]_{\Ri}, a_2]_{\Ri}, a_3]_{\Ri}\in F$. Thus
\begin{eqnarray*}
l_1(a_1)=[\pi, a_1]_{\Ri}, \quad l_2(a_1, a_2)=0, \quad l_3(a_1, a_2, a_3)=[[[\mu, a_1]_{\Ri}, a_2]_{\Ri}, a_3]_{\Ri}, \quad l_k=0, \quad k\geq 4,
\end{eqnarray*}
for $a_1, a_2, a_3\in F$.
\emptycomment{
\begin{eqnarray*}
&&l_1(a_1)(\mathfrak{X}_{1},\cdots, \mathfrak{X}_{k}, x)\\
&=&(-1)^k\Big(\sum_{i=1}^{k}\sum_{j\leq i}(-1)^{j} a_1\Big(\mathfrak{X}_{1},\cdots,\hat{\mathfrak{X}_{j}},\cdots, \mathfrak{X}_{i},
\pi\big(\mathfrak{X}_{j},x_{i+1}\big)\wedge y_{i+1},\mathfrak{X}_{i+2},\cdots,\mathfrak{X}_{k+1},x\Big)\\
&&\sum_{i=1}^{k}\sum_{j\leq i}(-1)^{j} a_1\Big(\mathfrak{X}_{1},\cdots,\hat{\mathfrak{X}_{j}},\cdots, \mathfrak{X}_{i},
x_{i+1}\wedge \pi\big(\mathfrak{X}_{j},y_{i+1}\big),\mathfrak{X}_{i+2},\cdots,\mathfrak{X}_{k+1},x\Big)\\
&&-\sum_{j=1}^{k+1}(-1)^{j}a_1\Big(\mathfrak{X}_{1},\cdots,\hat{\mathfrak{X}_{j}},\cdots,\mathfrak{X}_{k+1},
\pi(\mathfrak{X}_{j},x)\Big)\Big)
\end{eqnarray*}
}
\end{proof}

Now we are ready to realize $3$-Lie algebra morphisms as Maurer-Cartan elements.

\begin{thm}\label{Maurer-Cartanm}
Let $(\g, \pi)$ and $(\h, \mu)$ be $3$-Lie algebras. Then $f\in\Hom(\g, \h)$ is a $3$-Lie algebra morphism if and only if $f$ is a Maurer-Cartan element of the $L_\infty$-algebra $$(\oplus_{n=0}^{+\infty}\Hom(\underbrace{\wedge^{2}\g\otimes\cdots\otimes\wedge^{2}\g}_n\wedge\g,\h), l_1, l_3)$$ constructed in Proposition \ref{control}.
\end{thm}
\begin{proof}
For $f\in\Hom(\g, \h)$ and $x, y, z\in\g$, by Proposition \ref{control}, it follows that
\begin{eqnarray*}
&&l_1(f)(x\wedge y, z)+\frac{1}{6}l_3(f, f, f)(x\wedge y, z)\\
&=&[\pi, f]_{\Ri}(x\wedge y, z)+\frac{1}{6}[[[\mu, f]_{\Ri}, f]_{\Ri}, f]_{\Ri}(x\wedge y, z)\\
&=&-f(\pi(x, y, z))+\mu(f(x), f(y), f(z)).
\end{eqnarray*}
Thus $f\in\Hom(\g, \h)$ is a $3$-Lie algebra morphism if and only if $f$ is a Maurer-Cartan element of $(\oplus_{n=0}^{+\infty}\Hom(\underbrace{\wedge^{2}\g\otimes\cdots\otimes\wedge^{2}\g}_n\wedge\g,\h), l_1, l_3)$
\end{proof}

By Theorem \ref{twistLin}, it follows that $\Big(\oplus_{n=0}^{+\infty}\Hom(\underbrace{\wedge^{2}\g\otimes\cdots\otimes\wedge^{2}\g}_n\wedge\g,\h), \{l_k^{f}\}_{k=1}^{+\infty}\Big)$ is an $L_{\infty}$-algebra, where $l_k^{f}$ is given by
\begin{equation*}
l_k^{f}(x_1, \cdots, x_k)=\sum_{n=0}^{+\infty}\frac{1}{n!}l_{k+n}(\underbrace{f,\cdots,f}_{n},x_1,\cdots,x_k).
\end{equation*}
Then we have the following result.
\begin{cor}
Let $(\g, \pi), (\h, \mu)$ be $3$-Lie algebras, $f\in\Hom(\g, \h)$ be a $3$-Lie algebra morphism, and $f'\in\Hom(\g, \h)$ be a linear map. Then $f+f'$ is a $3$-Lie algebra morphism if and only if $f'$ is a Maurer-Cartan element of the twisted $L_\infty$-algebra $\Big(\oplus_{n=0}^{+\infty}\Hom(\underbrace{\wedge^{2}\g\otimes\cdots\otimes\wedge^{2}\g}_n\wedge\g,\h), \{l_k^{f}\}_{k=1}^{+\infty}\Big)$.
\end{cor}
\begin{proof}
By Theorem \ref{Maurer-Cartanm}, $f+f'$ is a $3$-Lie algebra morphism if and only if
$$
    \sum_{k=1}^{+\infty}\frac{1}{k!}l_k(f+f',\cdots,f+f')=0,
$$
 which is equivalent  to
   \begin{eqnarray*}
    \sum_{k=1}^{+\infty}\frac{1}{k!}l_k^{f}(f',\cdots,f')=0,
  \end{eqnarray*}
  i.e. $f'$ is a Maurer-Cartan element of  $\Big(\oplus_{n=0}^{+\infty}\Hom(\underbrace{\wedge^{2}\g\otimes\cdots\otimes\wedge^{2}\g}_n\wedge\g,\h), \{l_k^{f}\}_{k=1}^{+\infty}\Big)$.
\end{proof}

\section{Cohomologies of $3$-Lie algebra morphisms}\label{sec:coh}
In this section, we establish the cohomology theory of a $3$-Lie algebra morphism using Getzler's twisted $L_\infty$-algebra.
\begin{defi}\cite{Ka}
Let $(\g, \pi)$ be a $3$-Lie algebra and $V$ be a vector space. A representation of $(\g, \pi)$ on $V$ is a linear map $\rho:\wedge^{2}\g\lon\Hom(V, V)$ such that
\begin{eqnarray*}
\rho(x_1, x_2)\rho(x_3, x_4)&=&\rho(\pi(x_1, x_2, x_3), x_4)+\rho(x_3, \pi(x_1, x_2, x_4))+\rho(x_3, x_4)\rho(x_1, x_2),\\
\rho(x_1, \pi(x_2, x_3, x_4))&=&\rho(x_3, x_4)\rho(x_1, x_2)-\rho(x_2, x_4)\rho(x_1, x_3)+\rho(x_2, x_3)\rho(x_1, x_4),
\end{eqnarray*}
where $x_1, x_2, x_3, x_4\in\g$.
\end{defi}

\begin{ex}
Let $(\g, \pi)$ be a $3$-Lie algebra. Then $\ad:\wedge^{2}\g\lon\Hom(\g, \g)$ is a representation, where
$$\ad_{x\wedge y}z=\pi(x, y, z), \quad \forall x, y, z\in\g.$$
It is also called the adjoint representation.
\end{ex}
Let $\h$ be a $3$-Lie subalgebra of $\g$. Consider the following short exact sequence of vector spaces:
$$
0\longrightarrow \h\stackrel{i}{\longrightarrow}\g\stackrel{p}{\longrightarrow}\g/\h\longrightarrow0.
$$
Choose a section $s:\g/\h\lon\g$, i.e. $p\circ s=\Id$. Define a map $\overline{\ad}: \wedge^{2}\h\lon\Hom(\g/\h, \g/\h)$ by
$$
\overline{\ad}_{u\wedge v}(\overline{x})=\overline{\pi(u, v, s(\overline{x}))}, \quad \forall u, v\in\h, \bar{x}\in\g/\h.
$$
\begin{pro}\label{exad}
With the above notations, the map $\overline{\ad}: \wedge^{2}\h\lon\Hom(\g/\h, \g/\h)$ is independent of the choice of sections and is a representation of $\h$ on $\g/\h$.
\end{pro}
\begin{proof}
For another section $t:\g/\h\lon\g$, i.e. $p\circ t=\Id$, we have that $\mathrm{Im}(s-t)\in i(\h)$ and $\overline{\pi(u, v, s(\overline{x}))}-\overline{\pi(u, v, t(\overline{x}))}=\overline{\pi(u, v, (s-t)(\overline{x}))}=0$. Thus $\overline{\ad}: \wedge^{2}\h\lon\Hom(\g/\h, \g/\h)$ is independent of the choice of sections.

For any $u_1, u_2, u_3, u_4\in\h$ and $\overline{x}\in\g/\h$, we have that
\begin{eqnarray*}
\overline{\ad}_{u_1\wedge u_2}\overline{\ad}_{u_3\wedge u_4}(\overline{x})&=&\overline{\pi(u_1, u_2, s\Big(\overline{\pi(u_3, u_4, s(\overline{x}))}\Big))}=\overline{\pi\Big(u_1, u_2, \pi(u_3, u_4, s(\overline{x}))}\Big)\\
&=&\overline{\ad}_{\pi(u_1, u_2, u_3)\wedge u_4}(\overline{x})+\overline{\ad}_{u_3\wedge\pi(u_1, u_2, u_4)}(\overline{x})+\overline{\ad}_{u_3\wedge u_4}\overline{\ad}_{u_1\wedge u_2}(\overline{x}).
\end{eqnarray*}
Similarly, the following equation holds:
$$
\overline{\ad}_{u_1\wedge\pi(u_2, u_3, u_4)}=\overline{\ad}_{u_3\wedge u_4}\overline{\ad}_{u_1\wedge u_2}-\overline{\ad}_{x_2\wedge x_4}\overline{\ad}_{x_1\wedge x_3}+\overline{\ad}_{x_2\wedge x_3}\overline{\ad}_{x_1\wedge x_4}.
$$
Thus $\overline{\ad}: \wedge^{2}\h\lon\Hom(\g/\h, \g/\h)$ is a representation of $\h$ on $\g/\h$.
\end{proof}

Let $(V, \rho)$ be a representation of a $3$-Lie algebra $(\g, \pi)$. For $n\geq 1$, define the space of $n$-cochains $C^{n}(\g, V)$ to be
$$
C^{n}(\g, V)=\Hom(\underbrace{\wedge^{2}\g\otimes\cdots\otimes\wedge^{2}\g}_{n-1}\wedge\g, V).
$$
Define $\partial : C^{n}(\g, V)\lon C^{n+1}(\g, V)$ by
\begin{eqnarray*}
\label{eqdefp1}&&\partial(\theta)(x_1\wedge y_1, \cdots, x_n\wedge y_n, x_{n+1})\\
\nonumber&=&\sum_{1\leq i<j\leq n}(-1)^{i}\theta\Big(x_1\wedge y_1, \cdots, \widehat{x_i\wedge y_i}, \cdots, x_{j-1}\wedge y_{j-1}, \pi(x_i, y_i, x_j)\wedge y_j\\
\nonumber&&+x_j\wedge\pi(x_i, y_i, y_j), \cdots, x_{n}\wedge y_n, x_{n+1}\Big)\\
\nonumber&&+\sum_{i=1}^{n}(-1)^{i}\theta\Big(x_1\wedge y_1,\cdots, \widehat{x_i\wedge y_i}, \cdots, x_{n}\wedge y_n, \pi(x_i, y_i, x_{n+1})\Big)\\
\nonumber&&+\sum_{i=1}^{n}(-1)^{i+1}\rho(x_i, y_i)\theta(x_1\wedge y_1,\cdots, \widehat{x_i\wedge y_i}, \cdots, x_n\wedge y_n, x_{n+1})\\
\nonumber&&+(-1)^{n+1}\Big(\rho(y_n, x_{n+1})\theta(x_1\wedge y_1,\cdots, x_{n-1}\wedge y_{n-1}, x_{n})\\
&&+\rho(x_{n+1}, x_n)\theta(x_1\wedge y_1,\cdots, x_{n-1}\wedge y_{n-1}, y_{n})\Big).
\end{eqnarray*}
Then $(\oplus_{n=1}^{+\infty}C^{n}(\g, V), \rho, \partial)$ is a cochain complex. See \cite{Cas, Ta1} for more details.

\begin{defi}
The cohomology of the cochain complex $(\oplus_{n=1}^{+\infty}C^{n}(\g, V),  \partial)$ is called the cohomology of the $3$-Lie algebra $(\g, \pi)$ with coefficients in $V$. We denote the $n$-th cohomology group by $H^{n}(\g, V)$.
\end{defi}

In the following, we will introduce the cohomologies of a $3$-Lie algebra morphism via the twisted $L_\infty$-algebra governing its deformations.

Let $(\g,\pi), (\h, \mu)$ be $3$-Lie algebras and $f:\g\lon\h$ be a $3$-Lie algebra morphism. For $n\geq 0$, define the space of $n$-cochains $C^{n}(f)$ to be
\begin{equation}\label{eq:cochain}
C^{n}(f)=
\begin{cases}
(\g\wedge\g)\oplus(\h\wedge\h), & n=0;\\
\Hom(\underbrace{\wedge^{2}\g\otimes\cdots\otimes\wedge^{2}\g}_{n-1}\wedge\g,\h), & n\geq1.
\end{cases}
\end{equation}
Define the coboundary operator $\delta: C^{0}(f)\lon C^{1}(f)$ by
\begin{equation}
\delta(\huaX, \huaU)z=\sum_{i=1}^{k}\mu(u_i, v_i, f(z))-\sum_{j=1}^{l}f(\pi(x_j, y_j, z)),
\end{equation}
where $z\in\g, \huaU=\sum_{i=1}^{k}u_i\wedge v_i\in\wedge^{2}\h,~~\huaX=\sum_{j=1}^{l}x_j\wedge y_j\in\wedge^{2}\g$. Define $\delta: C^{n}(f)\lon C^{n+1}(f)$ by
\begin{equation}
\delta(\theta)=(-1)^{n-1}l_{1}^{f}(\theta)=(-1)^{n-1}(l_1(\theta)+\frac{1}{2}l_3(f, f, \theta)),
\end{equation}
where $\theta\in \Hom(\underbrace{\wedge^{2}\g\otimes\cdots\otimes\wedge^{2}\g}_{n-1}\wedge\g,\h)$.

\begin{thm}
  $(\oplus_{n=0}^{+\infty}C^{n}(f), \delta)$ is a cochain complex, i.e. $\delta\circ\delta=0$.
\end{thm}
\begin{proof}
Since $\Big(\oplus_{j=0}^{+\infty}\Hom(\underbrace{\wedge^{2}\g\otimes\cdots\otimes\wedge^{2}\g}_j\wedge\g,\h), \{l_{k}^{f}\}_{k=1}^{+\infty}\Big)$ is an $L_{\infty}$-algebra, then for all $\theta\in C^{n}(f)$, $n\geq 1$, we have
$$
\delta\circ\delta(\theta)=-l_{1}^{f}\circ l_{1}^{f}(\theta)=0.
$$

For all $(x\wedge y, u\wedge v)\in (\g\wedge\g)\oplus(\h\wedge\h)=C^0(f)$, since $f:\g\lon\h$ is a $3$-Lie algebra morphism, it follows that
\begin{eqnarray*}
&&\delta(\delta(x\wedge y, u\wedge v))(x_1\wedge y_1, z)\\
&=&l_1(\delta(u\wedge v))(x_1\wedge y_1, z)+\frac{1}{2}l_3(f, f, \delta(u\wedge v))(x_1\wedge y_1, z)\\
&=&-\delta(x\wedge y, u\wedge v)(\pi(x_1, y_1, z))+\mu(f(x_1), f(y_1), \delta(x\wedge y, u\wedge v)(z))\\
&&+\mu(\delta(x\wedge y, u\wedge v)(x_1), f(y_1), f(z))+\mu(f(x_1), \delta(x\wedge y, u\wedge v)(y_1), f(z))\\
&=&-\mu\Big(u, v, f(\pi(x_1, y_1, z))\Big)+\mu(f(x_1), f(y_1), \mu(u, v, f(z)))+\mu(\mu(u, v, f(x_1)), f(y_1), f(z))\\
&&+\mu(f(x_1), \mu(u, v, f(y_1)), f(z))+f(\pi(x, y, \pi(x_1, y_1, z)))-\mu(f(x_1), f(x_2), f(\pi(x, y, z)))\\
&&-\mu(f(\pi(x, y, x_1)), f(y_1), f(z))-\mu(f(x_1), f(\pi(x, y, y_1)), f(z))=0.
\end{eqnarray*}
Thus $\delta\circ\delta=0$, for $n\geq 0$.
\end{proof}

\begin{defi}
Let $f:\g\lon\h$ be a $3$-Lie algebra morphism. The cohomology of the cochain complex $(\oplus_{n=0}^{+\infty}C^{n}(f), \delta)$ is called the cohomology of the morphism $f$. We denote the set of $n$-cocycles by $Z^{n}(f)$, the set of $n$-coboundaries by $B^{n}(f)$, and  the $n$-th cohomology group by $H^{n}(f)=Z^{n}(f)/B^{n}(f)$.
\end{defi}

\begin{rmk}
  In~\cite{ABM}, Arfa, Ben Fraj and Makhlouf introduced the cohomology   of $n$-Lie algebra morphisms using a different approach in their study of formal deformations of $n$-Lie algebra morphisms.
\end{rmk}

In the sequel, we give the explicit formula of the coboundary operator $\delta$ as follows. For $\theta\in C^{n}(f), n\geq 1$, $\mathfrak{X}_{i}=x_i\wedge y_i\in\g\wedge\g$ and $x\in\g$, there is
\begin{eqnarray}
\label{eqdefde01}&&\delta(\theta)\Big(\mathfrak{X}_{1},\cdots, \mathfrak{X}_{n}, x\Big)\\
\nonumber&=&(-1)^{n-1}(l_1(\theta)+\frac{1}{2}l_3(f, f, \theta))\Big(\mathfrak{X}_{1},\cdots, \mathfrak{X}_{n}, x\Big)\\
\nonumber&=&\sum_{1\leq i<j\leq n}(-1)^{i}\Big(\theta(\mathfrak{X}_{1},\cdots, \hat{\mathfrak{X}_{i}}, \cdots, \mathfrak{X}_{j-1}, \pi(\mathfrak{X}_{i}, x_j)\wedge y_j, \cdots, \mathfrak{X}_{n}, x)\\
\nonumber&&+\theta(\mathfrak{X}_{1},\cdots, \hat{\mathfrak{X}_{i}}, \cdots, \mathfrak{X}_{j-1}, x_j\wedge\pi(\mathfrak{X}_{i}, y_j), \cdots, \mathfrak{X}_{n}, x)\Big)\\\nonumber&&+\sum_{i=1}^{n}(-1)^{i}\theta\Big(\mathfrak{X}_{1},\cdots, \hat{\mathfrak{X}_{i}}, \cdots, \mathfrak{X}_{n}, \pi(\mathfrak{X}_{i}, x)\Big)\\
\nonumber&&+\sum_{i=1}^{n}(-1)^{i+1}\mu\Big(f(x_i), f(y_i), \theta(\mathfrak{X}_{1},\cdots, \hat{\mathfrak{X}_{i}}, \cdots, \mathfrak{X}_{n}, x)\Big)\\
\nonumber&&+(-1)^{n-1}\Big(\mu(\theta(\mathfrak{X}_{1},\cdots, \mathfrak{X}_{n-1}, x_{n}), f(y_{n}), f(x))+\mu(f(x_{n}), \theta(\mathfrak{X}_{1},\cdots, \mathfrak{X}_{n-1}, y_{n}), f(x))\Big).
\end{eqnarray}
In particular, for $\theta\in C^{1}(f)=\Hom(\g, \h)$, the formula of $\delta(\theta)$ will appear frequently in the following sections and we list it as following
\begin{eqnarray}
\label{eqclosed}&&\delta(\theta)(x\wedge y, z)\\
\nonumber&=&-\theta(\pi(x, y, z))+\mu(f(x), f(y), \theta(z))+\mu(\theta(x), f(y), f(z))+\mu(f(x), \theta(y), f(z)).
\end{eqnarray}

At the end of this section, we show that there is a cochain map between the cochain complex of $3$-Lie algebra $(\h, \mu)$ with the adjoint representation and the cochain complex of a $3$-Lie algebra morphism.

Denote by $$(\oplus_{n=1}^{+\infty}C^{n}(\h, \h), \partial)=\Big(\oplus_{n\geq 1}\Hom(\underbrace{\wedge^{2}\h\otimes\cdots\otimes\wedge^{2}\h}_{n-1}\wedge\h,\h),  \partial\Big)$$ the cochain complex of $3$-Lie algebra $(\h, \mu)$ with coefficients in the adjoint representation. For any $3$-Lie algebra morphism $f:\g\lon\h$, define a map
$$
f^*:\Hom(\underbrace{\wedge^{2}\h\otimes\cdots\otimes\wedge^{2}\h}_{n-1}\wedge\h,\h)\lon\Hom(\underbrace{\wedge^{2}\g\otimes\cdots\otimes\wedge^{2}\g}_{n-1}\wedge\g,\h)
$$ by
$$
f^*(\theta)(x_1\wedge y_1, \cdots, x_{n-1}\wedge y_{n-1}, x)=\theta(f(x_1)\wedge f(y_1), \cdots, f(x_{n-1})\wedge f(y_{n-1}), f(x)).
$$
Then we have the following proposition.
\begin{pro}\label{procomap}
With the above notations, $f^*$ is a cochain map, i.e. we have the following commutative diagram
\[
\small{ \xymatrix{
 0\ar[d]^{f^*} \ar[r]^{\partial}&
\Hom(\h, \h) \ar[d]^{f^*} \ar[r]^{\quad\partial}&\cdots\ar[r]& \Hom(\underbrace{\wedge^{2}\h\otimes\cdots\otimes\wedge^{2}\h}_{n-1}\wedge\h,\h) \ar[d]^{f^*} \ar[r]  & \cdots  \\
(\g\wedge\g)\oplus(\h\wedge\h)\ar[r]^{\delta} & \Hom(\g, \h) \ar[r]^{\quad \delta} &\cdots\ar[r]& \Hom(\underbrace{\wedge^{2}\g\otimes\cdots\otimes\wedge^{2}\g}_{n-1}\wedge\g,\h)\ar[r]& \cdots.}
}
\]
Consequently, $f^*$ induces a homomorphism of cohomology groups, which is also denoted by $f^*: H^{n}(\h, \h)\lon H^{n}(f)$, for $n\geq 0$.
\end{pro}
\begin{proof}
For $n=0$, since $H^{0}(\h, \h)=0$, then $f^*:H^{0}(\h, \h)\lon H^{0}(f)$ is a group homomorphism. For any $\theta\in C^{n}(\h, \h), n\geq 1$, we have
\begin{eqnarray*}
&&f^*\Big(\partial(\theta)\Big)\Big(x_{1}\wedge y_1,\cdots, x_{n}\wedge y_n, x\Big)\\
&=&\partial(\theta)\Big(f(x_{1})\wedge f(y_1),\cdots, f(x_{n})\wedge f(y_n), f(x)\Big)\\
&=&\sum_{1\leq i<j\leq n}(-1)^{i}\theta\Big(f(x_1)\wedge f(y_1), \cdots, \widehat{f(x_i)\wedge f(y_i)}, \cdots, f(x_{j-1})\wedge f(y_{j-1}),\\
&& \mu(f(x_i), f(y_i), f(x_j))\wedge f(y_j)+f(x_j)\wedge\mu(f(x_i), f(y_i), f(y_j)), \cdots, f(x_{n})\wedge f(y_n), f(x)\Big)\\
&&+\sum_{i=1}^{n}(-1)^{i}\theta\Big(f(x_1)\wedge f(y_1),\cdots, \widehat{f(x_i)\wedge f(y_i)}, \cdots, f(x_{n})\wedge f(y_n), \mu(f(x_i), f(y_i), f(x))\Big)\\
&&+\sum_{i=1}^{n}(-1)^{i+1}\mu\Big(f(x_i), f(y_i), \theta(f(x_1)\wedge f(y_1),\cdots, \widehat{f(x_i)\wedge f(y_i)}, \cdots, f(x_n)\wedge f(y_n), f(x))\Big)\\
&&+(-1)^{n+1}\Big(\mu(f(y_n), f(x), \theta\Big(f(x_1)\wedge f(y_1),\cdots, f(x_{n-1})\wedge f(y_{n-1}), f(x_{n})\Big))\\
&&+\mu(f(x), f(x_n), \theta(f(x_1)\wedge f(y_1),\cdots, f(x_{n-1})\wedge f(y_{n-1}), f(y_{n})))\Big)\\
&=&\delta(f^*(\theta))\Big(x_{1}\wedge y_1,\cdots, x_{n}\wedge y_n, x\Big).
\end{eqnarray*}
Thus $f^*$ is a homomorphism from the cochain complex $(\oplus_{n=1}^{+\infty}C^{n}(\h, \h), \partial)$ to $(\oplus_{n=1}^{+\infty}C^{n}(f), \delta)$, which implies that $f^*: H^{n}(\h, \h)\lon H^{n}(f)$ is a group homomorphism for $n\geq 0$.
\end{proof}

\section{Rigidity and stability of $3$-Lie algebra morphisms}\label{sec:sec:rig-hom}
In this section, we study deformations of $3$-Lie algebra morphism. We show that the condition $H^1(f)=0$ will imply a rigidity theorem, and the condition $H^2(f)=0$ will imply a stability theorem.
\begin{defi}
Let $f$ be a $3$-Lie algebra morphism from $(\g, \pi)$ to $(\h, \mu)$. A deformation of $f$ is a smooth one parameter family of morphisms $f_t$ from $(\g, \pi)$ to $(\h, \mu)$ such that $f_0=f$.
\end{defi}

\begin{defi}
Two deformations $f_t$ and $f_t'$ of $f$ are called equivalent if there exists two smooth family of inner isomorphisms of $3$-Lie algebras $\varphi_t: \h\lon \h$ and $ \psi_t:\g\lon\g$ such that $\varphi_0=\Id, \psi_0=\Id$ and
  \begin{equation}\label{eq:equivalent}
  \varphi_t\circ f_{t}'=f_t\circ\psi_t.
  \end{equation}
 \end{defi}

Note that if  $\psi_t$ and $ \varphi_t$ are inner automorphisms of $3$-Lie algebras $\g$ and $\h$ respectively, then there are $\huaU=\sum_{i=1}^{k}u_i\wedge v_i\in\h\wedge\h$ and $\huaX=\sum_{j=1}^{l}x_j\wedge y_j\in\g\wedge\g$ such that
    $$
    \frac{d}{dt}|_{t=0}\varphi_t=\sum_{i=1}^{k}\mu(u_i, v_i, \cdot), \quad \frac{d}{dt}|_{t=0}\psi_t=\sum_{j=1}^{l}\pi(x_j, y_j, \cdot).
    $$

Let $f_t$ be a deformation of $f$. Denote $\frac{d}{dt}|_{t=0}f_t$ by $\dot{f_0}$.
\begin{pro}\label{protangentgamma}
With the above notations, $\dot{f_0}$ is a $1$-cocycle in $C^{1}(f)$. Moreover if $f_t$ and $f_t'$ are equivalent deformations of $f$, then $[\dot{f_0}]=[\dot{f_0'}]$ in $H^{1}(f)$.
\end{pro}
\begin{proof}
Since $f_t$ is a $3$-Lie algebra morphism, for any $x, y, z\in\g$, we have
\begin{eqnarray}
\label{eqR'}&&\mu(\dot{f_0}(x), f(y), f(z))+\mu(f(x), \dot{f_0}(y), f(z))+\mu(f(x), f(y), \dot{f_0}(x))\\
\nonumber&=&\frac{d}{dt}|_{t=0}\mu(f_t(x), f_t(y), f_t(z))=\frac{d}{dt}|_{t=0}f_t(\pi(x, y, z))\\
\nonumber&=&\dot{f_0}(\pi(x, y, z)).
\end{eqnarray}
Thus by \eqref{eqclosed} and \eqref{eqR'}, we have $\delta(\dot{f_0})=0$.

Assume that $\varphi_t:\h\lon\h$ and $\psi_t:\g\lon\g$ are inner isomorphisms such that $f_t$ and $f'_t$ are equivalent.
Denote $\frac{d}{dt}|_{t=0}\varphi_t$ by $\dot{\varphi_0}$ and $\frac{d}{dt}|_{t=0}\psi_t$ by $\dot{\psi_0}$. Taking derivative of \eqref{eq:equivalent}, and using the fact that $\varphi_0=\Id $ and  $\psi_0=\Id$,  we have
$$
\dot{\varphi_0}(f(x))+\dot{f_0'}(x)=\dot{f_0}(x)+f(\dot{\psi_0}(x)).
$$
Since $\varphi_t$ and $\psi_t$ are inner automorphisms of $3$-Lie algebras $\h$ and $\g$ respectively, it follows that there exist $\huaU=\sum_{i=1}^{k}u_i\wedge v_i\in\h\wedge\h$ and $\huaX=\sum_{j=1}^{l}y_j\wedge z_j\in\g\wedge\g$ such that $$\dot{\varphi_0}=\sum_{i=1}^{k}\mu(u_i, v_i, \cdot), \quad \dot{\psi_0}=\sum_{j=1}^{l}\pi(y_j, z_j, \cdot).$$ Therefore, we have
\begin{equation*}
\sum_{i=1}^{k}\mu(u_i, v_i, f(x))+\dot{f'_0}(x)=\dot{f_0}(x)+f(\sum_{j=1}^{l}\pi(y_j, z_j, x)),
\end{equation*}
which implies $\dot{f_0}-\dot{f'_0}=\delta(\huaX, \huaU)$. Thus $[\dot{f_0}]=[\dot{f_0'}]$ in $H^{1}(f)$.
\end{proof}

Next, we consider under what conditions does a cocycle $\alpha\in Z^{1}(f)$ determine a deformation $f_t$. To solve this problem, we need to introduce the following geometric tools.
\emptycomment{
Define the Kuranishi map $K: Z^{1}(f)\lon H^{2}(f)$ by
\begin{equation*}
K(\alpha)=[l^{f}_2(\alpha, \alpha)], \quad \forall \alpha\in Z^{1}(f).
\end{equation*}
More precisely,
\begin{eqnarray}
\label{defiCour}
l_2^{f}(\alpha, \alpha)(x\wedge y, z)&=&l_3(f, \alpha, \alpha)(x\wedge y, z)\\
\nonumber&=&[[[\mu, f]_{\Ri}, \alpha]_{\Ri}, \alpha]_{\Ri}(x\wedge y, z)\\
\nonumber&=&2\Big(\mu(\alpha(x), \alpha(y), f(z))+\mu(\alpha(x), f(y), \alpha(z))+\mu(f(x), \alpha(y), \alpha(z))\Big).
\end{eqnarray}

\begin{pro}\label{pro:K}
Let $(\g, \pi)$ and $(\h, \mu)$ be $3$-Lie algebras, $f:\g\lon\h$ be a $3$-Lie algebra morphism. Assume that there exists a deformation $f_t$ of $f$ such that $\dot{f_0}=\alpha\in Z^{1}(f)$, then $K(\alpha)=0$.
\end{pro}
\begin{proof}
Consider the Taylor expansion of $f_t$ around $t=0$, then we have
\begin{equation*}
f_t(x)=f(x)+t\alpha(x)+\frac{t^{2}}{2}g(x)+o(t^{3}).
\end{equation*}
Since $\mu(f_t(x), f_t(y), f_t(z))=f_t(\pi(x, y, z))$ and $\alpha\in Z^{1}(f)$, we have
\begin{equation}\label{Taylor}
\frac{t^{2}}{2}\Big(\delta(g)(x, y)+2\mu(\alpha(x), \alpha(y), f(z))+2\mu(\alpha(x), f(y), \alpha(z))+2\mu(f(x), \alpha(y), \alpha(z))\Big)+o(t^3)=0.
\end{equation}
Thus by \eqref{defiCour} and \eqref{Taylor}, we obtain $l_{2}^{f}(\alpha, \alpha)=-\delta g$, which implies that $K(\alpha)=0$.
\end{proof}

\yh{Is it true if $K(\alpha)=0$ then $\alpha\in Z^{1}(f)$ determine a deformation $f_t$?}
}

Let $E\stackrel{\pi}{\lon} M$ be a vector bundle. Assume that there is a smooth action $\cdot: G\times E\lon E$ of a Lie group $G$ on $E$ preserving the zero-section $Z: M\lon E$. It follows that $M$ inherits a $G$-action. We also denote the action of $G$ on $M$ by $\cdot: G\times M\lon M$. For all $x\in M$, define a smooth map $\zeta_x: G\lon M$ by $\zeta_x(g)=g\cdot x$. Denote the tangent map of $\zeta_x$ at the unit $e_G$ of $G$ by $D(\zeta_x)_{e_G}$, which is a linear map from $\g$ to $T_xM$.
\begin{defi}{\rm(\cite{CSS})}
A section $s: M\lon E$ is called {\bf equivariant} if $s$ satisfies
\begin{equation*}
s(g\cdot x)=g\cdot s(x), \quad \forall g\in G, x\in M.
\end{equation*}
\end{defi}

Denote the zero set of a section $s: M\lon E$ by $z(s)=\{x\in M| s(x)=0\}$. A zero point $x\in M$ of $s$ is called non-degenerate if the sequence
\begin{equation*}
\g\stackrel{D(\zeta_x)_{e_G}}{\longrightarrow}T_xM\stackrel{D^{v}(s)_x}{\longrightarrow} E_x
\end{equation*}
is exact, where $D^{v}(s)_x$ is the vertical derivative of $s$ at $x$.

\begin{pro}\rm(\cite{CSS})\label{proorbR}
Let $s$ be an equivariant section of the vector bundle $E\stackrel{\pi}{\lon} M$ and $x$ be a non-degenerate zero of $s$. Then there is an open neighborhood $U$ of $x$ and a smooth map $p: U\lon G$ such that for all $m\in U$ with $s(m)=0$, one has $p(m)\cdot x=m$. In particular, the orbit of $x$ under the action of $G$ and the zero set of $s$ coincide in an open neighborhood of $x$.
\end{pro}

\begin{pro}\rm(\cite{CSS})\label{proorbR1}
Let $E$ and $F$ be vector bundles over a smooth manifold $M$. Let $s\in \Gamma(E)$ be a section and $\phi\in\Gamma(\Hom(E, F))$ be a vector bundle map such that $\phi\circ s=0$. Suppose that $x\in M$ is $s(x)=0$ such that
\begin{equation*}
T_xM\stackrel{D^{v}(s)_x}{\longrightarrow} E_x\stackrel{\phi_x}{\longrightarrow} F_x
\end{equation*}
is exact. Then there are the following conclusions
\begin{itemize}
\item [\rm(i)]
$s^{-1}(0)$ is locally a manifold around $x$ of dimension $\mathrm{dim}\ker(D^{v}(s)_x)$.
\item [\rm(ii)]
If $s'$ is another section of $\Gamma(E)$ which is $\huaC^0$-close to $s$, and $\phi'$ is another vector bundle map $E\lon F$ which is $\huaC^{0}$-close to $\phi$ and such that $\phi'\circ s'=0$, then there exists $x'\in M$ close to $x$ such that $s'(x')=0$.
\end{itemize}
\end{pro}

Let $(\g\oplus\h, \pi+\mu)$ be the direct product $3$-Lie algebra of $(\g, \pi)$ and $(\h, \mu)$. Denote by
$$\InnDer(\g\oplus
 \h)=\{(\ad_{\huaX}, \ad_{\huaU})|~~\forall~~\huaX\in\wedge^2\g, \huaU\in\wedge^{2}\h\}.$$
It is obvious that $\InnDer(\g\oplus \h)$ is a Lie algebra.

Denote the connected, simply connected Lie group of $\InnDer(\g), \InnDer(\h)$ and $\InnDer(\g\oplus\h)$ by $\InnAut(\g), \InnAut(\h)$ and $\InnAut(\g\oplus\h)$ respectively, then $\InnAut(\g\oplus\h)=(\InnAut(\g), \InnAut(\h))$. Define an action of $\InnAut(\g\oplus\h)$ on $\Hom(\g, \h)$ by
\begin{equation*}
\cdot: \InnAut(\g\oplus\h)\times\Hom(\g, \h)\lon\Hom(\g, \h), \quad (A, B)\cdot g=BgA^{-1},
\end{equation*}for all $(A, B)\in\InnAut(\g\oplus\h), g\in\Hom(\g, \h),$ where $A\in \InnAut(\g)$ and $B\in \InnAut(\h)$.
Assume that $f:\g\lon\h$ is a $3$-Lie algebra morphism, then the orbit $\Orb_f=\{(A, B)\cdot f| (A, B)\in\InnAut(\g\oplus\h) \}$ of $f$ is a manifold. Define a map $\zeta_f: \InnAut(\g\oplus\h)\lon \Hom(\g, \h)$ by $\zeta_f(A, B)=(A, B)\cdot f$. Then $T_f\Orb_f$ is $D(\zeta_f)_e(\InnDer(\g\oplus\h))$, where $D(\zeta_f)_e$ is the tangent map of $\zeta_f$ at the unit $e$ of the Lie group $\InnAut(\g\oplus\h)$.

\begin{pro}\label{protang}
With the above notations, $T_f\Orb_f$ is $B^{1}(f)$.
\end{pro}
\begin{proof}
Since $T_f\Orb_f=D(\zeta_f)_e(\InnDer(\g\oplus\h))$, for any $q\in T_f\Orb_f$, there exists $\huaU=\sum_{i=1}^{k}u_i\wedge v_i\in\h\wedge\h$ and $\huaX=\sum_{j=1}^{l}x_j\wedge y_j$ such that
\begin{eqnarray*}
q&=&\frac{d}{dt}|_{t=0}\exp(t\ad_{\huaU})f\exp(-t\ad_{\huaX})=\sum_{i=1}^{k}\ad_{u_i\wedge v_i}f-\sum_{j=1}^{l}f\ad_{x_j\wedge y_j}=\delta(\huaX,\huaU).
\end{eqnarray*}
Thus we have $T_f\Orb_f=B^{1}(f)$.
\end{proof}

\begin{thm}\label{thmrigidity}{\bf (Rigidity of $3$-Lie algebra morphisms)}
Let $f:\g\lon\h$ be a $3$-Lie algebra morphism. If $H^{1}(f)=0$, then there exists an open neighborhood $U\subset \Hom(\g, \h)$ of $f$ and a smooth map $p: U\lon \InnAut(\h)$ such that $p(f')\cdot f=f'$ for every $3$-Lie algebra morphism $f'\in U$.
\end{thm}
\begin{proof}
Denote by $M=\Hom(\g, \h)$ and $E=\Hom(\g, \h)\times\Hom(\wedge^{3}\g, \h)$. Then $E$ is a trivial vector bundle over $M$ with fiber $\Hom(\wedge^{3}\g, \h)$. Define an action of $\InnAut(\g\oplus\h)$ on the manifold $E$ by
\begin{equation*}
\cdot: \InnAut(\g\oplus\h)\times E\lon E,
\end{equation*}
where
\begin{equation*}
(A, B)\cdot(g, \alpha)(x, y\wedge z\wedge w)=(B(g(A^{-1}(x))), B(\alpha(A^{-1}(y), A^{-1}(z), A^{-1}(w)))),
\end{equation*}
for $(g, \alpha)\in E,~(A, B)\in\InnAut(\g\oplus\h).$ Define a section $s: M\lon E$ by
\begin{equation*}
s(g)=(g, S(g)), \quad \forall g\in M,
\end{equation*}
where $S: \Hom(\g, \h)\lon\Hom(\wedge^{3}\g, \h)$ is given by
\begin{equation*}
S(g)(x, y, z)=\mu(g(x), g(y), g(z))-g(\pi(x, y, z)),
\end{equation*}for all $g\in\Hom(\g, \h), ~x, y, z\in\g.$
Then for any $(A,B)\in \InnAut(\g\oplus\h), g\in M$ and $x, y, z\in\g$, we have
\begin{eqnarray*}
(A,B)S(g)(x, y, z)&=&B(\mu(g(A^{-1}(x)), g(A^{-1}(y)), g(A^{-1}(z)))-g(\pi(A^{-1}(x), A^{-1}(y), A^{-1}(z))))\\
&=&\mu(Bg(A^{-1}(x)), Bg(A^{-1}(y)), Bg(A^{-1}(z)))-BgA^{-1}(\pi(x, y, z))\\
&=&S(BgA^{-1})(x, y, z).
\end{eqnarray*}
Thus we have
$$
(A,B)\cdot s(g)=s((A, B )\cdot g),
$$
which implies that $s$ is an equivariant section.

Since $f:\g\lon\h$ is a $3$-Lie algebra morphism, it follows that $f\in z(s)$. Moreover, since $E$ is a trivial vector bundle, we have $D^{v}(s)_f=D(S)_f: T_fM\lon E_f$. For any $g\in\Hom(\g, \h),$ we have
\begin{eqnarray}\label{defiDS}
&&D(S)_f(g)(x, y, z)\\
\nonumber&=&\frac{d}{dt}|_{t=0}S(f+tg)(x, y, z)\\
\nonumber&=&\frac{d}{dt}|_{t=0}\Big(\mu(f(x)+tg(x), f(y)+tg(y), f(z)+tg(z))-(f+tg)(\pi(x, y, z))\Big)\\
\nonumber&=&\mu(g(x), f(y), f(z))+\mu(f(x), g(y), f(z))+\mu(f(x), f(y), g(z))-g(\pi(x, y, z))\\
\nonumber&=&\delta(g)(x, y, z).
\end{eqnarray}

By Proposition \ref{protang} and $H^{1}(f)=0$, we have that $f$ is a non-degenerate zero of $s$. By Proposition \ref{proorbR}, there exists an open neighborhood $U\subset \Hom(\g, \h)$ of $f$ and a smooth map $p: U\lon \InnAut(\g\oplus\h)$ such that $p(f')\cdot f=f'$ for $3$-Lie algebra morphism $f'\in U$.
\end{proof}

Now, we give the sufficient condition on a $1$-cocycle to give a deformation of $f$.
\begin{thm}\label{thmnedef}
Let $f:\g\lon\h$ be a $3$-Lie algebra morphism. If $H^{2}(f)=0$, then the space of $3$-Lie algebra morphisms is a manifold in a neighborhood of $f$, whose  dimension is $\mathrm{dim}Z^{1}(f)$, and any $\alpha\in Z^{1}(f)$ gives rise to a deformation of $f$.
\end{thm}
\begin{proof}
Denote by $M=\Hom(\g, \h), E=\Hom(\g, \h)\times\Hom(\wedge^{3}\g, \h)$ and $F=\Hom(\g, \h)\times\Hom(\wedge^{2}\g\otimes\wedge^{2}\g\wedge\g, \h)$. Then $E$ and $F$ are trivial vector bundles over $M$ with fibers $\Hom(\wedge^{3}\g, \h)$ and $\Hom(\wedge^{2}\g\otimes\wedge^{2}\g\wedge\g, \h)$ respectively. Define a smooth map $\phi: E\lon F$ by
\begin{equation*}
\phi(g, \alpha)=(g, l_1(\alpha)+\frac{1}{2}l_3(g, g, \alpha)), \quad \forall g\in M, \alpha\in\Hom(\wedge^{3}\g, \h),
\end{equation*}
where $l_1, l_3$ are given by \eqref{eqabli}. Then $\phi$ is a vector bundle map.

Define a section $s: M\lon E$ by
\begin{equation*}
s(g)=(g, S(g)), \quad \forall g\in M,
\end{equation*}
where $S: \Hom(\g, \h)\lon\Hom(\wedge^{3}\g, \h)$ is given by
\begin{equation*}
S(g)(x, y, z)=\mu(g(x), g(y), g(z))-g(\pi(x, y, z)),
\end{equation*}for all $g\in\Hom(\g, \h), ~x, y, z\in\g.$

By Proposition \ref{control}, for any $g\in\Hom(\g, \h)$, it follows that
\begin{equation}\label{homojaco}
l_1(l_1(g))=0, \quad l_1(l_3(g, g, g))+3l_3(g, g, l_1(g))=0, \quad l_3(g, g, l_3(g, g, g))=0,
\end{equation}
where $l_1, l_3$ are given by \eqref{eqabli}. Thus we have $\phi\circ s(g)=(g, 0)$, which implies $\phi\circ s=0$. Moreover, assume that $s(f)=0$, it means that $f$ is a $3$-Lie algebra morphism. Denote by $\phi_f=\phi(f, \cdot): E_f\lon F_f$. Then $\phi_f=\delta$. By \eqref{defiDS} and $H^{2}(f)=0$, the following sequence
$$
T_fM\stackrel{D^{v}(s)_f}{\longrightarrow} E_f\stackrel{\phi_f}{\longrightarrow} F_f
$$
is exact. By Proposition \ref{proorbR1}, we obtain that the space $W$ of morphisms from $\g$ to $\h$ is a manifold in a neighborhood of $f$, whose dimension is $\mathrm{dim}Z^{1}(f)$.

Assume $\gamma(t)\in W$, by Proposition \ref{protangentgamma}, we have $\dot{\gamma}(0)\in Z^{1}(f)$. Since $\mathrm{dim}W=\mathrm{dim}Z^{1}(f)$, then $T_fW=Z^{1}(f)$. Thus any $\alpha\in Z^{1}(f)$ gives rise to a deformation of $f$.
\end{proof}
\emptycomment{
Now, we give the sufficient condition on a $1$-cocycle to give a deformation of $f$.
\begin{cor}\label{thmrigidity3}
With the above notations, if $H^{2}(f)=0$, then any $\alpha\in Z^{1}(f)$ gives rise to a deformation of $f$.
\end{cor}
\begin{proof}
Since $f$ is a morphism of $3$-Lie algebras and $H^{2}(f)=0$, we have that the space $W$ of morphisms of $3$-Lie algebras is a manifold in a neighborhood of $f$, whose dimension is $\mathrm{dim}Z^{1}(f)$. Assume $\gamma(t)\in W$, by Proposition \ref{protangentgamma}, we have $\dot{\gamma}(0)\in Z^{1}(f)$. Moreover, $\mathrm{dim}W=\mathrm{dim}Z^{1}(f)$, then $T_fW=Z^{1}(f)$. Thus any $\alpha\in Z^{1}(f)$ gives rise to a deformation of $f$.
\end{proof}
}
At the end of this section, we study the stability of $3$-Lie algebra morphisms.

\begin{defi}
Let $f$ be a $3$-Lie algebra morphism from $(\g, \pi)$ to $(\h, \mu)$. The $3$-Lie algebra morphism $f$ is called {\bf stable} if for any $3$-Lie algebra structure $\varpi$ on $\h$ that is $\huaC^{0}$-close to $\mu$, there exists a $3$-Lie algebra morphism $f'$ from $(\g, \pi)$ to $(\h, \varpi)$ which is $\huaC^{0}$-close to $f$.
\end{defi}

\begin{thm}\label{thmrigidity2}{\bf (Stability of $3$-Lie algebra morphisms)}
Let $f$ be a $3$-Lie algebra morphism from $(\g, \pi)$ to $(\h, \mu)$. If $H^2(f)=0$, then $f$ is stable.
\end{thm}
\begin{proof}
Let $M=\Hom(\g, \h), E=\Hom(\g, \h)\times\Hom(\wedge^{3}\g, \h)$ and $F=\Hom(\g, \h)\times\Hom(\wedge^{2}\g\otimes\wedge^{2}\g\wedge\g, \h)$. Then $E$ and $F$ are trivial vector bundles over $M$ with fiber $\Hom(\wedge^{3}\g, \h)$ and $\Hom(\wedge^{2}\g\otimes\wedge^{2}\g\wedge\g, \h)$ respectively.

Denote by $\Omega$ the space of all $3$-Lie algebra structures on $\h$. For any $\varpi\in\Omega$, there is a section $s_\varpi\in\Gamma(E)$ given by
$$
s_\varpi(g)=(g, S_\varpi(g)), \quad\text{where} \quad S_\varpi(g)(x, y, z)=\varpi(g(x), g(y), g(z))-g(\pi(x, y, z)), \quad \forall x, y, z\in\g,
$$
and a bundle map $\phi_\varpi: E\lon F$,
$$
\phi_{\varpi}(g, \alpha)=(g, [\pi, \alpha]_{\Ri}+\frac{1}{2}[[[\varpi, g]_{\Ri}, g]_{\Ri}, \alpha]_{\Ri}).
$$
Moreover, the map
$$
\Omega\lon\Gamma(E)\times\Gamma(\Hom(E, F)), \quad \varpi\lon (s_{\varpi}, \phi_{\varpi}),
$$
is continuous under the $\huaC^{1}$-topology on the space of sections.

We apply Proposition \ref{control} with $\Delta=\pi+\varpi$. By \eqref{eqabli}, we have
$$
l_1(\alpha)=[\pi, \alpha]_{\Ri}, \quad l_3(g, g, \alpha)=[[[\varpi, g]_{\Ri}, g]_{\Ri}, \alpha]_{\Ri}, \quad S_\varpi(g)=l_1(g)+\frac{1}{6}l_3(g, g, g).
$$
Since
\begin{equation*}
l_1(l_1(g))=0, \quad l_1(l_3(g, g, g))+3l_3(g, g, l_1(g))=0, \quad l_3(g, g, l_3(g, g, g))=0,
\end{equation*}
then $\phi_{\varpi}\circ s_{\varpi}(g)=(g, 0)$ for all $\varpi\in \Omega$.

By \eqref{defiDS} and $H^{2}(f)=0$, it shows that
$$
T_fM\stackrel{D^{v}(s_\mu)_f}{\longrightarrow} E_f\stackrel{(\phi_{\mu})_f}{\longrightarrow} F_f
$$
is exact. By Proposition \ref{proorbR1}, for any $\varpi$ sufficiently close to $\mu$, there exists a $f'\in M$ close to $f$ such that $f'$ is a $3$-Lie algebra morphism from $(\g, \pi)$ to $(\h, \varpi)$.
\end{proof}

\section{Deformations of $3$-Lie subalgebras and stability}\label{sec:rig-sub}
In this section, we study deformations of $3$-Lie subalgebras using the established cohomology theory. We show that the condition $H^2(\h, \g/\h)=0$ will imply the space of $3$-Lie subalgebras is a manifold in a neighborhood of $\h$. We also show a stability theorem of a $3$-Lie subalgebra under the condition $H^2(\h, \g/\h)=0$.

Let $(\g, \pi)$ be a $3$-Lie algebra and $\h\subset\g$ be a $k$-dimensional $3$-Lie subalgebra of $(\g, \pi)$. Denote by $(\oplus_{n=1}^{+\infty}C^{n}(\h, \g/\h), \partial)=\Big(\oplus_{n\geq 1}\Hom(\underbrace{\wedge^{2}\h\otimes\cdots\otimes\wedge^{2}\h}_{n-1}\wedge\h,\g/\h), \partial\Big)$ the cochain complex of the $3$-Lie algebra $(\h, \pi)$ with coefficients in the representation $(\g/\h, \overline{\ad})$, where $\overline{\ad}$ is constructed in Proposition \ref{exad}. We denote the set of $n$-cocycles by $Z^{n}(\h, \g/\h)$, the $n$-th cohomology group by $H^{n}(\h, \g/\h)$. 
\emptycomment{
\begin{eqnarray*}
&&\partial(\theta)(u_1\wedge v_1, \cdots u_k\wedge v_n, u_{n+1})\\
&=&\sum_{1\leq i<j\leq n}(-1)^{i}\theta\Big(u_1\wedge v_1, \cdots, \widehat{u_i\wedge v_i}, \cdots, u_{j-1}\wedge v_{j-1}, \pi(u_i, v_i, u_j)\wedge v_j\\
&&+u_j\wedge\pi(u_i, v_i, v_j), \cdots, u_{n}\wedge v_n, u_{n+1}\Big)+\sum_{i=1}^{n}(-1)^{i}\theta\Big(u_1\wedge v_1,\cdots, \widehat{u_i\wedge v_i}, \cdots, u_{n}\wedge v_n, \pi(u_i, v_i, u_{n+1})\Big)\\
&&+\sum_{i=1}^{n}(-1)^{i+1}\overline{\ad}(u_i, v_i)\theta(u_1\wedge v_1,\cdots, \widehat{u_i\wedge v_i}, \cdots, u_n\wedge v_n, u_{n+1})\\
&&+(-1)^{n+1}\Big(\overline{\ad}(v_n, u_{n+1})\theta(u_1\wedge v_1,\cdots, u_{n-1}\wedge v_{n-1}, u_{n})+\overline{\ad}(u_{n+1}, u_n)\theta(u_1\wedge v_1,\cdots, u_{n-1}\wedge v_{n-1}, v_{n})\Big).
\end{eqnarray*}
}

Denote by $\mathrm{Gr}_{k}(\g)$ the Grassmannian manifold of $k$-dimensional subspaces of $\g$, i.e.
$$
\mathrm{Gr}_{k}(\g)=\{U\subset\g|U~~\text{is a}~~k-\text{dimensional subspace of}~~\g\}.
$$
\begin{defi}
Let $(\g, \pi)$ be a $3$-Lie algebra and $\h\subset\g$ be a $k$-dimensional $3$-Lie subalgebra of $(\g, \pi)$. A deformation of $\h$ in $\g$ is a smooth curve $\h_t$ in $\mathrm{Gr}_{k}(\g)$ such that $\h_0=\h$ and $\h_t$ is a $3$-Lie subalgebra of $\g$ for all $t$.
\end{defi}

Since $\h_t$ is a smooth curve in $\mathrm{Gr}_{k}(\g)$ such that $\h_0=\h$, there exists a curve $A(t)\in\GL(\g)$ such that $A(t)\h=\h_t$. Denote by $T_\h\mathrm{Gr}_{k}(\g)$ the tangent space of $\mathrm{Gr}_{k}(\g)$ at $\h$, it is known that  $T_\h\mathrm{Gr}_{k}(\g)\cong\Hom(\h, \g/\h)$. Define $\dot{\h_0}:\h\lon\g/\h$ by $$\dot{\h_0}(u)=\frac{d}{dt}|_{t=0}A(t)u~~\mathrm{mod}~~\h.$$

\begin{pro}\label{protangentgamma11}
With the above notations, $\dot{\h_0}$ is a $1$-cocycle in $C^{1}(\h, \g/\h)$, i.e. $[\dot{\h_0}]$ in $H^{1}(\h, \g/\h)$.
\end{pro}
\begin{proof}
For any $u, v, w\in\h$, define $\gamma_t\in\Hom(\wedge^{3}\h, \h)$ by $$\gamma_t(u, v, w)=(A(t))^{-1}\pi(A(t)u, A(t)v, A(t)w)-\pi(u, v, w).$$ It obvious that $\gamma_t(u, v, w)\in\h$ and $\gamma_0(u, v, w)=0$. Since $$\frac{d}{dt}|_{t=0}A(t)\gamma_t(u, v, w)=\frac{d}{dt}|_{t=0}\gamma_t(u, v, w)+\frac{d}{dt}|_{t=0}A(t)\gamma_0(u, v, w)=\frac{d}{dt}|_{t=0}\gamma_t(u, v, w)\in\h,$$
thus we have $\frac{d}{dt}|_{t=0}A(t)\gamma_t(u, v, w)~~\mathrm{mod}~~\h=0$. On the other hand,
\begin{eqnarray*}
&&\frac{d}{dt}|_{t=0}A(t)\gamma_t(u, v, w)\\
&=&\frac{d}{dt}|_{t=0}\Big(\pi(A(t)u, A(t)v, A(t)w)-A(t)\pi(u, v, w)\Big)\\
&=&\pi(\frac{d}{dt}|_{t=0}A(t)u, v, w)+\pi(u, \frac{d}{dt}|_{t=0}A(t)v, w)+\pi(u, v, \frac{d}{dt}|_{t=0}A(t)w)-\frac{d}{dt}|_{t=0}A(t)\pi(u, v, w),
\end{eqnarray*}
which implies that
\begin{eqnarray*}
&&\overline{\ad}_{u\wedge v}(\dot{\h_0}(w))+\overline{\ad}_{v\wedge w}(\dot{\h_0}(u))+\overline{\ad}_{w\wedge u}(\dot{\h_0}(v))-\dot{\h_0}(\pi(u, v, w))\\
&=&\pi(\frac{d}{dt}|_{t=0}A(t)u, v, w)~~\mathrm{mod}~~\h+\pi(u, \frac{d}{dt}|_{t=0}A(t)v, w)~~\mathrm{mod}~~\h+\pi(u, v, \frac{d}{dt}|_{t=0}A(t)w)~~\mathrm{mod}~~\h\\
&&-\frac{d}{dt}|_{t=0}A(t)\pi(u, v, w)~~\mathrm{mod}~~\h=0.
\end{eqnarray*}
Thus $\partial(\dot{\h_0})(u, v, w)=0$, i.e. $[\dot{\h_0}]\in H^{1}(\h, \g/\h)$.
\end{proof}

\begin{thm}\label{thm11111}
 Let $(\g, \pi)$ be a $3$-Lie algebra and $\h\subset\g$ be a $k$-dimensional $3$-Lie subalgebra of $(\g, \pi)$. If $H^{2}(\h, \g/\h)=0$, then the space of $3$-Lie subalgebras is a manifold in a neighborhood of $\h$, whose dimension is $\mathrm{dim}Z^{1}(\h, \g/\h)$. Moreover, if $H^{2}(\h, \g/\h)=0$, any $\alpha\in Z^{1}(\h, \g/\h)$ gives rise to a deformation of $\h$.
\end{thm}
\begin{proof}
Let $M=\mathrm{Gr}_{k}(\g)$ and $E=\coprod\limits_{V\in M}\Hom(\wedge^{3}V, \g/V), F=\coprod\limits_{V\in M}\Hom(\wedge^{2}V\otimes\wedge^{2}V\wedge V, \g/V)$. Then $E$ and $F$ are vector bundles over $M$ whose fibres of $V$ are $\Hom(\wedge^{3}V, \g/V)$ and $\Hom(\wedge^{2}V\otimes\wedge^{2}V\wedge V, \g/V)$ respectively.

Consider the following exact sequence of vector spaces,
$$
0\longrightarrow U\stackrel{i}{\longrightarrow}\g\stackrel{p}{\longrightarrow}\g/U\longrightarrow0,
$$
Fix an inner product on $\g$, then $\g=U\oplus U^\bot$ and denote by $x=x_U+x_{U^{\bot}}$. There is a section $s:\g/U\lon\g$, i.e. $p\circ s=\Id$ as following
$$
s([x])=x_{U^{\bot}}, \quad \forall [x]\in\g/U.
$$
Define a smooth map $\phi: E\lon F$ by
\begin{equation*}
\phi(U, \alpha)=(U, \tilde{\partial}(\alpha)), \quad \forall U\in M, \alpha\in\Hom(\wedge^{3}U, \g/U),
\end{equation*}
where $\tilde{\partial}:\Hom(\wedge^{3}U, \g/U)\lon\Hom(\wedge^{2}U\otimes\wedge^{2}U\wedge U, \g/U)$ is given by
\begin{eqnarray*}
&&\tilde{\partial}(\alpha)(u_1\wedge v_1, u_2\wedge v_2, w)\\
&=&-\alpha\Big((\pi(u_1, v_1, u_2)-sp(\pi(u_1, v_1, u_2)))\wedge v_2+u_2\wedge(\pi(u_1, v_1, v_2)-sp(\pi(u_1, v_1, v_2))), w\Big)\\
&&-\alpha(u_2\wedge v_2, \pi(u_1, v_1, w)-sp(\pi(u_1, v_1, w)))+\alpha(u_1\wedge v_1, \pi(u_2, v_2, w)-sp(\pi(u_2, v_2, w)))\\
&&+p\Big(\pi(u_1, v_1, s(\alpha(u_2\wedge v_2, w))-\pi(u_2, v_2, s(\alpha(u_1\wedge v_1, w)))\\
&&-\pi(v_2, w, s(\alpha(u_1\wedge v_1, u_2)))-\pi(w, u_2, s(\alpha(u_1\wedge v_1, v_2)))\Big).
\end{eqnarray*}
Then $\phi$ is a vector bundle map.

Define a section $\sigma: M\lon E$ by
\begin{equation*}
\sigma(U)=(U, \Lambda(U)), \quad \forall U\in M,
\end{equation*}
where $\Lambda(U)\in\Hom(\wedge^{3}U, \g/U)$ is given by
\begin{equation}\label{equationlam}
\Lambda(U)(u, v, w)=p(\pi(u, v, w)),\quad \forall U\in M, ~u, v, w\in U.
\end{equation}
Since $(\g, \pi)$ is a $3$-Lie algebra, we have that
\begin{eqnarray*}
&&\tilde{\partial}(\Lambda(U))(u_1\wedge v_1, u_2\wedge v_2, w)\\
&=&-p\Big(\pi(\pi(u_1, v_1, u_2)-sp(\pi(u_1, v_1, u_2)), v_2, w)\Big)-p\Big(\pi(u_2, \pi(u_1, v_1, v_2)-sp(\pi(u_1, v_1, v_2)), w)\Big)\\
&&-p\Big(\pi(u_2, v_2, \pi(u_1, v_1, w)-sp(\pi(u_1, v_1, w)))\Big)+p\Big(\pi(u_1, v_1, \pi(u_2, v_2, w)-sp(\pi(u_2, v_2, w)))\Big)\\
&&+p\Big(\pi(u_1, v_1, sp(\pi(u_1, v_1, w)))\Big)-p\Big(\pi(u_2, v_2, sp(\pi(u_1, v_1, w)))\Big)\\
&&-p\Big(\pi(v_2, w, sp(\pi(u_1, v_1, u_2)))\Big)-p\Big(\pi(w, u_2, sp(\pi(u_1, v_1, v_2)))\Big)\\
&=&0.
\end{eqnarray*}
Thus $\phi(\sigma(U))=(U, 0)$, i.e. $\phi\circ\sigma=0$.

Moreover, by \eqref{equationlam}, for $\h\in M$, we have that $\sigma(\h)=0$ if and only if $\h$ is a $3$-Lie subalgebra of $\g$. Denote $\phi_\h: E_\h\lon F_\h$ by $\phi(\h, \cdot)$, then $\phi_\h$ is the coboundary operator of $3$-Lie algebra $(\h, \pi)$ with coefficients in the representation $\g/\h$. For any $a\in\Hom(\h,\g/\h)\cong T_{\h}M$, there exists $A(t)\in\GL(\g)$ such that $a(u)=p(\frac{d}{dt}|_{t=0}A(t)u)$, we have
\begin{eqnarray*}
D^{v}(\sigma)_\h(a)(u, v, w)&=&\frac{d}{dt}|_{t=0}A(t)^{*}\Big(\Lambda(A(t)\h)\Big)(u, v, w)\\
&=&p(\frac{d}{dt}|_{t=0}A(t)^{-1}\Big(\pi(A(t)u, A(t)v, A(t)w)\Big))\\
&=&p\Big(-a(\pi(u, v, w))+\pi(a(u), v, w)+\pi(u, a(v), w)+\pi(u, v, a(w))\Big)\\
&=&\partial(a)(u, v, w),
\end{eqnarray*}
where $A(t)^{*}:\Hom(\h_t, \g/\h_{t})\lon\Hom(\h, \g/\h)$ is given by
$$
A(t)^{*}(f_t)(u)=[A(t)^{-1}\Big(f_t(A(t)u)\Big)], \quad \forall~~f_t\in\Hom(\g/\h_t, \g/\h_t), u\in\h.
$$
Since $H^{2}(\h, \g/\h)=0$, it follows that
$$
T_{\h}M\stackrel{D^{v}(\sigma)_\h}{\longrightarrow} E_\h\stackrel{\phi_\h}{\longrightarrow} F_\h
$$
is exact. By Proposition \ref{proorbR1}, we obtain that the space $H$ of $3$-Lie subalgebras is a manifold in a neighborhood of $\h$, whose dimension is $\mathrm{dim}Z^{1}(\h, \g/\h)$.

Moreover, assume $\gamma(t)\in H$, by Proposition \ref{protangentgamma11}, we have $\dot{\gamma}(0)\in Z^{1}(\h, \g/\h)$. Since $\mathrm{dim}H=\mathrm{dim}Z^{1}(\h, \g/\h)$, then it means that $T_{\h}H=Z^{1}(\h, \g/\h)$. Thus any $\alpha\in Z^{1}(\h, \g/\h)$ gives rise to a deformation of $\h$.
\end{proof}

\begin{defi}
 Let $(\g, \pi)$ be a $3$-Lie algebra and $\h\subset\g$ be a $k$-dimensional $3$-Lie subalgebra of $(\g, \pi)$. The $\h$ is called stable if for any $3$-Lie algebra structure $\pi'$ on $\g$ is $\huaC^{0}$-close to $\pi$, there exists a $k$-dimensional $3$-Lie subalgebra $\h'\subset\g$ which is $\huaC^{0}$-close to $\h$.
\end{defi}

\begin{thm}\label{staoflies}{\bf (Stability of $3$-Lie subalgebra)}
 Let $(\g, \pi)$ be a $3$-Lie algebra and $\h\subset\g$ be a $k$-dimensional $3$-Lie subalgebra of $(\g, \pi)$. If $H^{2}(\h, \g/\h)=0$, then $\h$ is a stable subalgebra of $\g$.
\end{thm}
\begin{proof}
Let $M=\mathrm{Gr}_{k}(\g)$ and $E=\coprod\limits_{V\in M}\Hom(\wedge^{3}V, \g/V), F=\coprod\limits_{V\in M}\Hom(\wedge^{2}V\otimes\wedge^{2}V\wedge V, \g/V)$. Then $E$ and $F$ are vector bundles over $M$ whose fibres of $V$ are $\Hom(\wedge^{3}V, \g/V)$ and $\Hom(\wedge^{2}V\otimes\wedge^{2}V\wedge V, \g/V)$ respectively.

Denote by $\Omega$ the space of all $3$-Lie algebra structures on $\g$. For any $\varpi\in\Omega$, there is a section $\sigma_\varpi\in\Gamma(E)$ given by
$$
\sigma_\varpi(V)=(V, \Lambda_\varpi(V)), \quad\text{where} \quad \Lambda_\varpi(V)(u, v, w)=p(\varpi(u, v, w)), \quad \forall u, v, w\in V,
$$
and a bundle map $\phi_\varpi: E\lon F$,
$$
\phi_{\varpi}(V, \alpha)=(V, \tilde{\partial}_{\varpi}(\alpha)), \quad \forall V\in M, \alpha\in\Hom(\wedge^{3}V, \g/V),
$$
where $\tilde{\partial}_{\varpi}:\Hom(\wedge^{3}V, \g/V)\lon\Hom(\wedge^{2}V\otimes\wedge^{2}V\wedge V, \g/V)$ is given by
\begin{eqnarray*}
&&\tilde{\partial}_{\varpi}(\alpha)(u_1\wedge v_1, u_2\wedge v_2, w)\\
&=&-\alpha\Big((\varpi(u_1, v_1, u_2)-sp(\varpi(u_1, v_1, u_2)))\wedge v_2+u_2\wedge(\varpi(u_1, v_1, v_2)-sp(\varpi(u_1, v_1, v_2))), w\Big)\\
&&-\alpha(u_2\wedge v_2, \varpi(u_1, v_1, w)-sp(\varpi(u_1, v_1, w)))+\alpha(u_1\wedge v_1, \varpi(u_2, v_2, w)-sp(\varpi(u_2, v_2, w)))\\
&&+p\Big(\varpi(u_1, v_1, s(\alpha(u_2\wedge v_2, w))-\varpi(u_2, v_2, s(\alpha(u_1\wedge v_1, w)))\\
&&-\varpi(v_2, w, s(\alpha(u_1\wedge v_1, u_2)))-\varpi(w, u_2, s(\alpha(u_1\wedge v_1, v_2)))\Big).
\end{eqnarray*}
Moreover, the map
$$
\Omega\lon\Gamma(E)\times\Gamma(\Hom(E, F)), \quad \varpi\lon (\sigma_{\varpi}, \phi_{\varpi}),
$$
is continuous under the $\huaC^{1}$-topology on the space of sections.

It follows that $\phi_{\varpi}\circ \sigma_{\varpi}(V)=(V, 0)$ for all $\varpi\in \Omega$. Assume that $\sigma_\pi(\h)=(\h, 0)$, i.e. $\h$ is a $3$-Lie subalgebra of $(\g, \pi)$, similar with the proof of Theorem \ref{thm11111}, it shows that
$$
T_{\h}M\stackrel{D^{v}(\sigma_\pi)_\h}{\longrightarrow} E_\h\stackrel{(\phi_{\pi})_\h}{\longrightarrow} F_\h
$$
is exact. By Proposition \ref{proorbR1}, for any $\varpi$ is sufficiently close to $\pi$, there exists a $\h'\in M$ close to $\h$ such that $\h'$ is a $3$-Lie subalgebra of $(\g, \varpi)$.
\end{proof}

At the end of this section, we will consider the relation between deformations of $3$-Lie algebra morphisms and deformations of $3$-Lie subalgebras.

\begin{pro}\label{deformal}
Let $(\g, \pi)$ and $(\h, \mu)$ be $3$-Lie algebras. If $f_t:\g\lon\h$ is a deformation of the $3$-Lie algebra morphism $f:\g\lon\h$, then there is a deformation of the $3$-Lie subalgebra $\mathrm{G}_f=\{(x, f(x))|\forall x\in\g\}$ in $\g\oplus\h$, where the $3$-Lie algebra structure on $\g\oplus\h$ is given by
$$
[(x, u), (y, v), (z, w)]_{\g\oplus\h}=\Big(\pi(x, y, z), \mu(u, v, w)\Big), \quad \forall x, y, z\in\g, u, v, w\in\h.
$$
\end{pro}
\begin{proof}
Since $f_t$ is a deformation of $f:\g\lon\h$, then $\mathrm{G}_{f_t}$ is a curve in $\mathrm{Gr}_{\mathrm{dim}\g}(\g\oplus\h)$ and $\mathrm{G}_{f_t}$ is a $3$-Lie subalgebra of $\g\oplus\h$ with $\mathrm{G}_{f_0}=\mathrm{G}_{f}$. Thus $\mathrm{G}_{f_t}$ is a deformation of $\mathrm{G}_{f}$ in $\g\oplus\h$.
\end{proof}

On the cohomology groups level, we have the following proposition, which implies that the rigidity and the stability of $3$-Lie algebra morphisms and subalgebras of $3$-Lie algebras are insistent.
\begin{pro}
Let $(\g, \pi)$ and $(\h, \mu)$ be $3$-Lie algebras, $f:\g\lon\h$ be a $3$-Lie algebra morphism. Then $H^k(f)\cong H^k\Big(\mathrm{G}_{f}, (\g\oplus\h)/\mathrm{G}_f\Big)$ for $k\geq 1$.
\end{pro}
\begin{proof}
For any $n\geq 1$, define maps
$$
\Xi_{n}:\Hom(\underbrace{\wedge^{2}\g\otimes\cdots\otimes\wedge^{2}\g}_{n-1}\wedge\g,\h)\lon\Hom\Big(\underbrace{\wedge^{2}\mathrm{G}_f\otimes\cdots\otimes\wedge^{2}\mathrm{G}_{f}}_{n-1}\wedge\mathrm{G}_f,(\g\oplus\h)/\mathrm{G}_f\Big)
$$
by
\begin{eqnarray*}
&&\Xi_n(\alpha)\Big((x_1, f(x_1))\wedge(y_1, f(y_1)), \cdots, (x_{n-1}, f(x_{n-1})), (x, f(x))\Big)\\
&=&[(0, \alpha(x_1\wedge y_1, \cdots, x_{n-1}\wedge y_{n-1}, x))].
\end{eqnarray*}
Since $\Xi_{n}$ are injective and $\mathrm{dim}(C^{n}(f))=\mathrm{dim}(C^{n}(\mathrm{G}_f, (\g\oplus\h)/\mathrm{G}_f)$, then $\Xi_n$ are isomorphic for $n\geq 1$.

Moreover, since $\overline{\ad}_{(x, f(x))\wedge(y, f(y))}[(z, u)]=[(0, \mu(f(x), f(y), u-f(z)))]$, we have
\begin{eqnarray*}
&&\partial(\Xi_n(\alpha))\Big((x_1, f(x_1))\wedge(y_1, f(y_1)), \cdots, (x_{n}, f(x_{n})), (x, f(x))\Big)\\
&=&[\Big(0, \sum_{1\leq i<j\leq n}(-1)^{i}\alpha\Big(x_1\wedge y_1, \cdots, \widehat{x_i\wedge y_i}, \cdots, x_{j-1}\wedge y_{j-1}, \pi(x_i, y_i, x_j)\wedge y_j\\
&&+x_j\wedge\pi(x_i, y_i, y_j), \cdots, x_{n}\wedge y_n, x\Big)+\sum_{i=1}^{n}(-1)^{i}\alpha\Big(x_1\wedge y_1,\cdots, \widehat{x_i\wedge y_i}, \cdots, x_{n}\wedge y_n, \pi(x_i, y_i, x)\Big)\Big)]\\
&&+[\Big(0, \sum_{i=1}^{n}(-1)^{i+1}\mu(f(x_i), f(y_i), \alpha(x_1\wedge y_1,\cdots, \widehat{x_i\wedge y_i}, \cdots, x_n\wedge y_n, x))\\
&&+(-1)^{n+1}\Big(\mu(f(y_n), f(x), \alpha(x_1\wedge y_1,\cdots, x_{n-1}\wedge y_{n-1}, x_{n})\\
&&+\mu(f(x), f(x_n), \theta(x_1\wedge y_1,\cdots, x_{n-1}\wedge y_{n-1}, y_{n}))\Big)\Big)]\\
&=&\Xi_{n+1}(\delta(\alpha))\Big((x_1, f(x_1))\wedge(y_1, f(y_1)), \cdots, (x_{n}, f(x_{n})), (x, f(x))\Big).
\end{eqnarray*}
Thus $H^k(f)\cong H^k\Big(\mathrm{G}_f, (\g\oplus\h)/\mathrm{G}_f\Big)$ for $k\geq 1$.
\end{proof}
\emptycomment{
\begin{rmk}
In Proposition \ref{deformal}, it shows that deformations of the trivial morphism of $3$-Lie algebras $f_t:\g\lon\h$ induce deformations of $3$-Lie subalgebra $\g: \g_t\subset\g\oplus\h$. On the cohomology groups level, there are $H^k(f)\cong H^k(\g, (\g\oplus\h)/\g)$ for $k\geq 1$.
\end{rmk}
}

\vspace{2mm}
\noindent
{\bf Acknowledgements. } This research is supported by NSFC(12471060, 12401076) and China Postdoctoral Science Foundation (2023M741349).

 \end{document}